\documentclass[12pt]{article}
\usepackage{amsmath,amscd,amsthm,amssymb}

\newtheorem{theorem}{Theorem}[section]
\newtheorem{lemma}[theorem]{Lemma}
\newtheorem{corollary}[theorem]{Corollary}
\newtheorem{definition}[theorem]{Definition}
\newtheorem{remark}[theorem]{Remark}

\numberwithin{equation}{section}

\usepackage[]{graphicx}

\def\R{{\mathbb{R}}}
\def\Rn{{\mathbb{R}^n}}
\def\G{{\mathbb{G}}}

\def\a {\alpha}
\def\b {\beta}
\def\i{\infty}

\def\G{\mathbb{G}}

\def\Lloc{L^1_{\rm loc}(\G)}

\def\dual{\,^{^{\complement}}\!}

\def\dual{\,^{^{\complement}}\!}

\usepackage{amsmath,amscd,amsthm,amssymb}
\usepackage{color}

\begin{document}

\begin{center}
\Large Characterizations of Lipschitz functions via the commutators of maximal function in Orlicz spaces on stratified Lie groups
\end{center}

\

\centerline{\large V.S. Guliyev$^{a,}$
	\footnote{E-mail address: vagif@guliyev.com}}

\

\centerline{$^{a}$\it Institute of Applied Mathematics, Baku State University}
\centerline{\it Baku, AZ 1148 Azerbaijan}
\

\begin{abstract}
 We give necessary and sufficient conditions for the boundedness of the maximal commutators $M_{b}$, the commutators of the maximal operator $[b, M]$ and the commutators of the sharp maximal operator $[b, M^{\sharp}]$ in Orlicz spaces $L^{\Phi}(\mathbb{G})$ on any stratified Lie group $\mathbb{G}$ when $b$ belongs to Lipschitz spaces $\dot{\Lambda}_{\beta}(\G)$. We obtain some new characterizations for certain subclasses of Lipschitz spaces $\dot{\Lambda}_{\beta}(\G)$.
\end{abstract}

\



\section{Introduction}

The aim of this paper is to study the maximal commutators $M_{b}$, the commutators of the maximal operator $[b, M]$ and the commutators of the sharp maximal operator $[b, M^{\sharp}]$ in the Orlicz spaces $L^{\Phi}(\mathbb{G})$ on any stratified Lie group $\mathbb{G}$.

Let $\mathbb{G}$ be a stratified Lie group, $f\in \Lloc $ and $0 \le \a < Q$, where $Q$ is the homogeneous dimension of $\G$. The fractional maximal function $M_{\a} f$ is  defined by
\begin{equation*}
M_{\a}f(x)=\sup_{B \ni x} |B|^{-1+\frac{\a}{Q}} \, \int_{B} |f(y)|dy,
\end{equation*}
and the sharp maximal function of Fefferman and Stein $M^{\sharp} f$ is  defined by
\begin{equation*}
	M^{\sharp}f(x)=\sup_{B \ni x} |B|^{-1} \, \int_{B} |f(y) - f_{B}|dy,
\end{equation*}
where the supremum is taken over all balls $B \subset \G$ containing $x$, and $|B|$ is the Haar measure of the $\G$-ball $B$.

The fractional maximal commutator generated by $b\in L^1_{\rm loc}(\G)$ and $M_{\a}$ is defined by
\begin{equation*}
M_{b,\a}(f)(x)=\sup_{B \ni x} |B|^{-1+\frac{\alpha}{Q}} \, \int_{B} |b(x)-b(y)||f(y)|dy.
\end{equation*}

The commutators generated by $b\in L^1_{\rm loc}(\G)$ and $M_{\a}$, $M^{\sharp}$ are defined by
\begin{align*}
[b,M_{\a}] f(x)&=b(x) M_{\a} f(x) - M_{\a}(bf)(x)
\end{align*}
and
\begin{align*}
[b,M^{\sharp}] f(x)&=b(x) M^{\sharp} f(x) - M^{\sharp}(bf)(x).
\end{align*}

When $\a = 0$, we simply denote by $[b,M]=[b,M_{0}]$ and $M_{b}=M_{b,0}$.

In 1978, Janson \cite{Janson} gave some characterizations of the Lipschitz space $\dot{\Lambda}_{\beta}(\Rn)$ via commutator $[b, T]$ and the author proved that $b \in \dot{\Lambda}_{\beta}(\Rn)$ if and only if $[b, T]$ is bounded from $L^p(\Rn)$ to
$L^q(\Rn)$, where $1<p<n/\b$, $1/p-1/q=\b/n$ and $T$ is the classical singular integral operator (see also \cite{Pal1995}).
The mapping properties of $M_{b,\a}$, $[b,M_{\a}]$ and $[b,M^{\sharp}]$ have been studied extensively by many authors. See, for instance, \cite{AkbBurGul2020, BasMilRui, CruzFio, GarHarSegTor, GulDerHasMN, GulDerHasAMP, ZhangWu2009, ZhWuSun}. The operator $M_{b}:=M_{b,0}$ plays an important role in the study of commutators of singular integral operators with $BMO$ symbols (see, for instance, \cite{AgcGogKMus, AykArmOm2016, GarHarSegTor, HuYang2009, SegTorrPM1991}). The operators $M_{\a}$, $[b,M_{\a}]$, $M_{b,\a}$ and $[b,M^{\sharp}]$ play an important role in real and harmonic analysis and applications (see, for example \cite{BonfLanUg,FolSt,DerGulHasPis2018,GulEkKayaSaf,Stein93,ZhWuSun,ZhangAMP2019}).
The nonlinear commutator of maximal function $[b,M]$ can be used in studying the product of a function in $H_1$ and a function in $BMO$ (see \cite{BonIwanJZ2007} for instance).
Note that, the boundedness of the operator $M_{b}$ on $L^p(\Rn)$ spaces was proved by Garcia-Cuerva et al. \cite{GarHarSegTor}. In \cite{BasMilRui} by Bastero et al. studied the necessary and sufficient condition for the boundedness of $[b,M]$ on $L^p(\Rn)$ spaces. In \cite{ZhangWu2009} by Zhang and Wu considered the same problem for $[b,M_{\a}]$, see also \cite{ZhWuSun}.

In \cite{GulResMath2022} recently gave necessary and sufficient conditions for the boundedness of the fractional maximal commutators in the Orlicz spaces $L^{\Phi}(\mathbb{G})$ on any stratified Lie group $\G$ when $b$ belongs to $BMO(\G)$ spaces, and was obtained some new characterizations for certain subclasses of $BMO(\G)$ spaces. Stratified groups appear in quantum physics and many parts of
mathematics, including Fourier analysis, several complex variables, geometry, and topology \cite{BonfLanUg,FolSt,VarSaloffCou}. The geometric structure of stratified Lie groups is so good that they inherit many analysis properties from the Euclidean spaces \cite{Graf2009,Stein93}. Apart from this, the difference between the geometric structures of Euclidean spaces and stratified Lie groups makes the study of the function spaces on them more complicated. However, the study of Orlicz spaces on stratified Lie groups is quite a few, which makes it deserve a further investigation. In 2017, Zhang \cite{Zh2017} considered some new characterizations of the Lipschitz spaces $\dot{\Lambda}_{\beta}(\Rn)$ via the boundedness of maximal commutator $M_{b}$ and the (nonlinear) commutator $[b,M]$ in Lebesgue spaces
and Morrey spaces on Euclidean spaces.

Inspired by the above literature, our main aim is to characterize the commutator functions $b$, involved in the boundedness on Orlicz spaces of the maximal commutator $M_{b}$ (Theorems \ref{CommMaxCharOr017}, \ref{CommFrMaxCharOr01W}), the commutator of the fractional maximal operator $[b,M]$ (Theorem \ref{dnndk}) and the commutator of the sharp maximal operator $[b, M^{\sharp}]$ (Theorem \ref{dnndkCommSharpMax}). Actually, such a characterization was done in \cite[Theorems 4.5, 4.9,  4.13]{GulDerHasAMP} in the Eucledian case $\G=\Rn$.

By $A \lesssim B$ we mean that $A \le C B$ with some positive constant $C$
independent of appropriate quantities. If $A \lesssim B$ and $B \lesssim A$, we
write $A\approx B$ and say that $A$ and $B$ are  equivalent.

\section{Notations}

$~~~$ We first recall some preliminaries concerning stratified Lie groups
(or so-called Carnot groups). We refer the reader to the books \cite{BonfLanUg, FolSt, Stein93} for analysis on stratified Lie groups.
Let $\mathcal{G}$ be a finite-dimensional, stratified, nilpotent Lie algebra.
Assume that there is a direct sum vector space decomposition
\begin{equation}\label{1X}
\mathcal{G}=V_1\oplus\cdots\oplus V_m
\end{equation}
so that each element of $V_j$, $2\leq j\leq m$, is a linear combination of
($j-1$)th order commutator of elements of $V_1$. Equivalently,
\eqref{1X} is a stratification provided $\left[V_i,V_j\right]=V_{i+j}$
whenever $i+j\leq m$ and $\left[V_i,V_j\right] =0$ otherwise.
Let $X=\{X_1,\ldots,X_n\}$ be a basis for $V_1$ and $X_{ij},$
$1\le i\le k_j,$ for $V_j$ consisting of commutators of length $j$.
We set $X_{i1}=X_{i},$ $i=1,\ldots,n$ and $k_1=n$,
and we call $X_{i1}$ a commutator of length $1$.

If $\G$ is the simply connected Lie group associated with $\mathcal{G}$,
then the exponential mapping is a global diffeomorphism from $\mathcal{G}$ to $\G$.
Thus, for each $g \in \G$, there is $x=(x_{ij})\in\mathbb{R}^N,$
$1\le i\le k_j,$ $1\le j\le m$, $N=\sum\limits_{j=1}^{m} k_j$,
such that $g=\exp \left(\sum x_{ij} X_{ij}\right)$.
A homogeneous norm function $|\cdot|$ on $\G$ is defined by
$|g|=\left(\sum |x_{ij}|^{2\cdot m!/j}\right)^{1/(2\cdot m!)}$, and
$Q=\sum\limits_{j=1}^{m} j k_j$ is said to be the {\it homogeneous dimension}
of $\G$, since $d(\delta_r x)=r^Q dx$ for $r>0$.
The dilation $\delta_{r}$ on $\G$ is defined by
\begin{align*}
\delta_{r}(g)=\exp\left(\sum r^{j}x_{ij}X_{ij}\right)\quad\mbox{if}
\quad g=\exp\left(\sum x_{ij}X_{ij}\right).
\end{align*}

Since $\G$ is nilpotent, the exponential map is diffeomorphism from
$\G$ onto $\G$ which takes Lebesgue measure on $\G$ to a biinvariant Haar measure
{\it dx} on $\G$. The group identity of $\G$ will be referred to as the origin and denoted by $e$.

A homogenous norm on $\G$ is a continuous function $x\rightarrow\rho(x)$
from $\G$ to $[0,\infty)$, which is $C^\infty$ on $\G\backslash\{0\}$ and
satisfies $\rho(x^{-1})=\rho(x)$, $\rho(\delta_t x)=t\rho(x)$ for all
$x\in \G$, $t>0$; $\rho(e)=0$ (the group identity). Moreover,
there exists a constant $c_0 \geq 1$ such that $\rho(xy)\le c_0\left(\rho(x)+\rho(y)\right)$
for all $x,y\in \G$.
With this norm, we define the $\G$-ball centered at $x$ with radius $r$ by $B(x,r)=\{y\in\G:\rho(y^{-1} x)
<r\}$, and we denote by $B_r=B(e,r)=\{y \in \G ~:~\rho(y)< r\}$ the open ball centered at $e$,
the identity element of $\G$, with radius $r$. By ${\dual B}(x,r)=\G \setminus B(x,r)$
we denote the complement of $B(x,r)$. One easily recognizes that there exists $c_1=c_1(\G)$ such that
\begin{align*}
|B(x, r)|=c_1 \; r^{Q},~~ x \in \G, ~ r>0.
\end{align*}
The most basic partial differential operator in a stratified Lie group is the sub-Laplacian
associated with $X$ is the second-order partial differential operator on $\G$ given by
$\mathcal{L}=\sum_{i=1}^{n} X_i^2.$

First, we recall the definition of Young functions.
\begin{definition}\label{def2} A function $\Phi : [0,\infty) \rightarrow [0,\infty]$ is called a Young function if $\Phi$ is convex, left-continuous, $\lim\limits_{r\rightarrow +0} \Phi(r) = \Phi(0) = 0$ and $\lim\limits_{r\rightarrow \infty} \Phi(r) = \infty$.
\end{definition}
From the convexity and $\Phi(0) = 0$ it follows that any Young function is increasing.
If there exists $s \in  (0,\infty )$ such that $\Phi(s) = \infty $,
then $\Phi(r) = \infty $ for $r \geq s$.
The set of  Young  functions such that
\begin{equation*}
0<\Phi(r)<\infty \qquad \text{for} \qquad 0<r<\infty
\end{equation*}
will be denoted by  $\mathcal{Y}.$
If $\Phi \in  \mathcal{Y}$, then $\Phi$ is absolutely continuous on every closed interval in $[0,\infty )$
and bijective from $[0,\infty )$ to itself.

For a Young function $\Phi$ and  $0 \leq s \leq \infty $, let
$$\Phi^{-1}(s)=\inf\{r\geq 0: \Phi(r)>s\}.$$
If $\Phi \in  \mathcal{Y}$, then $\Phi^{-1}$ is the usual inverse function of $\Phi$.

It is well known that
\begin{equation}\label{2.3}
r\leq \Phi^{-1}(r)\widetilde{\Phi}^{-1}(r)\leq 2r \qquad \text{for } r\geq 0,
\end{equation}
where $\widetilde{\Phi}(r)$ is defined by
\begin{equation*}
\widetilde{\Phi}(r)=\left\{
\begin{array}{ccc}
\sup\{rs-\Phi(s): s\in  [0,\infty )\}
& , & r\in  [0,\infty ), \\
\infty &,& r=\infty .
\end{array}
\right.
\end{equation*}

A Young function $\Phi$ is said to satisfy the $\Delta_2$-condition, written $\Phi \in  \Delta_2$, if
$\Phi(2r)\le C\Phi(r)$, $r>0$ for some $C>1$. If $\Phi \in  \Delta_2$, then $\Phi \in  \mathcal{Y}$. A Young function $\Phi$ is said to satisfy the $\nabla_2$-condition, denoted also by  $\Phi \in  \nabla_2$, if $\Phi(r)\leq \frac{1}{2C}\Phi(Cr)$, $r\geq 0$ for some $C>1$.

\begin{definition} (Orlicz Space). For a Young function $\Phi$, the set
	$$L^{\Phi}(\G)=\left\{f\in  L^1_{\rm loc}(\G): \int _{\G}\Phi(k|f(x)|)dx<\infty
	\text{ for some $k>0$  }\right\}$$
	is called Orlicz space. The  space $L^{\Phi}_{\rm loc}(\G)$ is defined as the set of all functions $f$ such that  $f\chi_{_B}\in  L^{\Phi}(\G)$ for all balls $B \subset \G$.
\end{definition}
$L^{\Phi}(\G)$ is a Banach space with respect to the norm
$$
\|f\|_{L^{\Phi}(\G)}=\inf\left\{\lambda>0:\int _{\G}\Phi\Big(\frac{|f(x)|}{\lambda}\Big)dx
\leq 1\right\}.
$$

If $\Phi(r)=r^{p},\, 1\le p<\infty $, then $L^{\Phi}(\G)=L^{p}(\G)$. If $\Phi(r)=0$, $0\le r\le 1$ and $\Phi(r)=\infty$, $r> 1$, then $L^{\Phi}(\G)=L^{\infty}(\G)$.

For a measurable set $D \subset \G$, a measurable function $f$ and $t>0$, let
$
m(D,\ f,\ t)=|\{x\in D:|f(x)|>t\}|.
$
In the case $D=\G$, we shortly denote it by $m(f,\ t)$.
\begin{definition} The weak Orlicz space
	$$
	WL^{\Phi}(\G)=\{f\in L^{1}_{\rm loc}(\G):\Vert f\Vert_{WL^{\Phi}}<\infty\}
	$$
	is defined by the norm
	$$
	\Vert f\Vert_{WL^{\Phi}}=\inf\Big\{\lambda>0\ :\ \sup_{t>0}\Phi(t)m\Big(\frac{f}{\lambda},\ t\Big)\ \leq 1\Big\}.
	$$
\end{definition}
We note that $\|f\|_{WL^{\Phi}}\le \|f\|_{L^{\Phi}}$.

The following analogue of the H\"older's inequality is well known (see, for example, \cite{RaoRen}).
\begin{theorem}\label{HolderOr}
	Let $D\subset\G$ be a measurable set and $f, g$ be measurable functions on $D$. For a Young function $\Phi$ and its complementary function  $\widetilde{\Phi}$,
	the following inequality is valid
	\begin{equation}\label{lemHoldX}
	\int_{D}|f(x)g(x)|dx \leq 2 \|f\|_{L^{\Phi}(D)} \|g\|_{L^{\widetilde{\Phi}}(D)}.
	\end{equation}
\end{theorem}

By elementary calculations we have the following property.
\begin{lemma}\label{charorlc}
Let $\Phi$ be a Young function and $D$ be a set in $\G$ with finite Haar measure. Then
	\begin{equation*}
	\|\chi_{_D}\|_{L^{\Phi}(\G)} = \|\chi_{_D}\|_{WL^{\Phi}(\G)}=\frac{1}{\Phi^{-1}\left(|D|^{-1}\right)}.
	\end{equation*}
\end{lemma}
By Theorem \ref{HolderOr}, Lemma \ref{charorlc} and \eqref{2.3} we get the following estimate.
\begin{lemma} \label{ShVM01}
	For a Young function $\Phi$ and any $\G$-ball $B$, the following inequality is valid
	\begin{equation}\label{lemHold}
	\int_{B}|f(y)|dy \leq 2 |B| \Phi^{-1}\left(|B|^{-1}\right) \|f\|_{L^{\Phi}(B)}.
	\end{equation}
\end{lemma}

\

\section{Characterization of Lipschitz spaces via commutators}

$~~~$ For a given $\G$-ball $B$ and $0 \le \a < Q$, we define the following maximal function:
\begin{equation*}
	M_{\a,B}f(x)=\sup_{B\supseteq B' \ni x} |B'|^{-1+\frac{\a}{Q}}\int_{B'} |f(y)|dy,
\end{equation*}
where the supremum is taken over all balls $B'$ such that $x \in B' \subseteq B$.
Moreover, we denote by $M_{B}=M_{0,B}$ when $\a=0$.

In order to prove our main theorem, we also need the following lemma.
\begin{lemma}\label{estFrMax} $\cite{GulResMath2022}$ Let $0\le\a<Q$, and $f:\G\to \R$ be a locally integrable function.

$(1)$	If $B_0$ is a ball on $\G$, then $|B_0|^{\frac{\a}{Q}}\leq  M_{\a}\big(\chi_{B_0}\big)(x) = M_{\a,B_0}\big(\chi_{B_0}\big)(x)$ for every $x\in B_0$.

$(2)$ $M_{\a}\big(f \,\chi_{B}\big)(x) = M_{\a,B}(f)(x)$ and $M_{\a}\big(\chi_{B}\big)(x) = M_{\a,B}\big(\chi_{B}\big)(x) = |B|^{\frac{\a}{Q}}$ for every
$x \in B \subset \G$.
\end{lemma}

The following result completely characterizes the boundedness of $M_{\a}$ on Orlicz spaces.
\begin{theorem}\label{AdamsFrMaxCharOrl}   $\cite{GulResMath2022}$
	Let $0< \a<Q$, $\Phi, \Psi$ be Young functions and $\Phi\in\mathcal{Y}$. The condition
	\begin{equation}\label{adRieszCharOrl2}
	r^{\a} \Phi^{-1}\big(r^{-Q}\big) \le C \Psi^{-1}\big(r^{-Q}\big)
	\end{equation}
	for all $r>0$, where $C>0$ does not depend on $r$, is necessary and sufficient for the boundedness of $M_{\a}$ from $L^{\Phi}(\G)$ to $WL^{\Psi}(\G)$. Moreover, if $\Phi\in\nabla_2,$ the condition \eqref{adRieszCharOrl2} is necessary and sufficient for the boundedness of $M_{\a}$ from $L^{\Phi}(\G)$ to $L^{\Psi}(\G)$.
\end{theorem}

\begin{remark}
	Note that Theorem \ref{AdamsFrMaxCharOrl} in the case $\G = \Rn$ were proved in \cite{GulDerHasAMP}.
\end{remark}

In this section, as an application of Theorem \ref{AdamsFrMaxCharOrl} we consider the boundedness of $M_{b,\a}$ on Orlicz spaces when $b$ belongs to the Lipschitz space, by which some new characterizations of the Lipschitz spaces are given.

Next we give the definition of the Lipschitz spaces on $\G$, and state some basic properties and useful lemmas.
\begin{definition}\label{deflip}    (Lipschitz-type spaces on $\G$) Let $0 < \beta < 1$.\\
$(1)$ We say a function $b$ belongs to the Lipschitz space $\dot{\Lambda}_{\beta}(\G)$ if there exists a constant $C$ such that for all $x, y \in \G$,
$$
|b(x)-b(y)|\le C \rho(y^{-1}x)^{\beta}.
$$
The smallest such constant $C$ is called the $\dot{\Lambda}_{\beta}(\G)$ norm of $b$ and is denoted by $\|b\|_{\dot{\Lambda}_{\beta}(\G)}$.

$(2)$ $\cite{Krantz1982}$  The space ${\rm Lip}_{\beta}(\G)$ is defined to be the set of all locally integrable functions $b$, i.e., there exists a positive constant $C$, such that
$$
\sup_{B} \frac{1}{|B|^{1+\beta/Q}} \int_{B}|b(x)-b_{B}|dx \le C,
$$
where the supremum is taken over every ball $B \subset \G$ containing $x$ and $b_{B}=\frac{1}{|B|}\int_{B} b(y)dy$. The smallest such constant $C$ is called the ${\rm Lip}_{\beta}(\G)$ norm of $b$ and is denoted by $\|b\|_{{\rm Lip}_{\beta}(\G)}$.
\end{definition}

To prove the theorems, we need auxiliary results. The first one is the following characterizations
of Lipschitz space (see \cite{Krantz1982}).
\begin{lemma}\label{CharLipSp}
Let $0 < \beta < 1$ and $b\in L^{1}_{\rm loc}(\G)$, then

$(1)$
$$
\|b\|_{\dot{\Lambda}_{\beta}(\G)} \thickapprox \|b\|_{{\rm Lip}_{\beta}(\G)}.
$$

$(2)$ Let $B_1 \subset B_2 \subset \G$ and $b \in {\rm Lip}_{\beta}(\G)$, where $B_1$ and $B_2$ are balls. Then there exists a constant $C$ depends only on $B_1$ and $B_2$ such that
$$
|b_{B_1}-b_{B_2}| \le C \, \|b\|_{{\rm Lip}_{\beta}(\G)} \, |B_2|^{\frac{\beta}{Q}}.
$$

$(3)$ There exists a constant $C$ depends only on $\beta$ such that
$$
|b(x) - b(y)| \le C \, \|b\|_{{\rm Lip}_{\beta}(\G)} \, |B_2|^{\frac{\beta}{Q}}
$$
holds for any ball $B$ containing $x$ and $y$.
\end{lemma}

\begin{lemma}\label{pwslip}
Let $0<\beta<1$ and $b\in \dot{\Lambda}_{\beta}(\G)$. Then the following pointwise estimate holds
$$
M_{b}f(x) \leq C \|b\|_{\dot{\Lambda}_{\beta}(\G)} \, M_{\b}f(x).
$$
\end{lemma}
\begin{proof}
If $b\in \dot{\Lambda}_{\beta}(\G)$, then
\begin{align*}
M_{b}(f)(x) &=\sup\limits_{B\ni x}|B|^{-1}\int _{B} |b(x)-b(y)||f(y)|dy
\\
&\leq C \|b\|_{\dot{\Lambda}_{\beta}(\G)} \sup\limits_{B\ni x}|B|^{-1+\frac{\b}{Q}}
\int _{B} |f(y)|dy
\\
& = C\|b\|_{\dot{\Lambda}_{\beta}(\G)} \, M_{\b}f(x).
\end{align*}
\end{proof}

\begin{lemma}\label{estFrMaxCom}
	If $b\in  L^1_{\rm loc}(\G)$ and $B_0:=B(x_0,r_0)$, then
	$$
	|b(x)-b_{B_0}|\leq M_{b} \chi_{B_0}(x) ~~ \mbox{for every} ~~ x\in B_0.
	$$
\end{lemma}
\begin{proof}
For $x\in B_0$, we get
\begin{align*}
& M_{b} \chi_{B_0}(x)=\sup\limits_{B\ni x}|B|^{-1}\int _{B} |b(x)-b(y)|\chi_{B_0}(y)dy
\\
&=\sup\limits_{B\ni x}|B|^{-1}\int _{B\cap B_0} |b(x)-b(y)|dy
\geq |B_0|^{-1}\int _{B_0\cap B_0} |b(x)-b(y)|dy
\\
&\geq \big| |B_0|^{-1}\int _{B_0} (b(x)-b(y))dy\big| \, |b(x)-b_{B_0}|.
\end{align*}
\end{proof}

The following theorem is valid.
\begin{theorem} \label{CommMaxCharOr017}
	Let $0<\beta<1$, $b\in  L^1_{\rm loc}(\G)$, $\Phi, \Psi$ be Young functions and $\Phi\in\mathcal{Y}$.
	
	$1.~$ If $\Phi\in \nabla_2$ and the condition
	\begin{equation}\label{adFrCharOrl1M}
		t^{-\frac{\b}{Q}}\Phi^{-1}(t)\le C \, \Psi^{-1}(t),
	\end{equation}
	holds for all $t>0$, where $C>0$ does not depend on $t$, then the condition $b\in \dot{\Lambda}_{\beta}(\G)$
	is sufficient for the boundedness of $M_{b}$ from $L^{\Phi}(\G)$ to $L^{\Psi}(\G)$.
	
	$2.~$ If the condition
	\begin{equation}\label{adFrCharOrl1Mncs}
		\Psi^{-1}(t) \leq C \, \Phi^{-1}(t)t^{-\frac{\b}{Q}},
	\end{equation}
	holds for all $t>0$, where $C>0$ does not depend on $t$, then the condition $b\in \dot{\Lambda}_{\beta}(\G)$
	is necessary for the boundedness of $M_{b}$ from $L^{\Phi}(\G)$ to $L^{\Psi}(\G)$.
	
	$3.~$ If $\Phi\in \nabla_2$ and $\Psi^{-1}(t) \thickapprox \Phi^{-1}(t)t^{-\frac{\b}{Q}}$,
	then the condition $b\in \dot{\Lambda}_{\beta}(\G)$ is necessary and sufficient for the boundedness of $M_{b}$ from $L^{\Phi}(\G)$ to $L^{\Psi}(\G)$.
\end{theorem}
\begin{proof}
	(1) The first statement of the theorem follows from Theorem \ref{AdamsFrMaxCharOrl} and Lemma \ref{pwslip}.
	
	(2) We shall now prove the second part. Suppose that $\Psi^{-1}(t) \lesssim \Phi^{-1}(t)t^{-\frac{\b}{Q}}$
	and $M_{b,\a}$ is bounded from $L^{\Phi}(\G)$ to $L^{\Psi}(\G)$. Choose any ball $B$ in $\G$, by Lemmas \ref{charorlc} and \ref{ShVM01}
	\begin{align*}
		\frac{1}{|B|^{1+\frac{\b}{Q}}} \int_{B}|b(y)-b_{B}|dy & = \frac{1}{|B|^{1+\frac{\b}{Q}}}
		\int_{B} \Big| \int_{B} (b(y)-b(z))dz \Big| dy \notag
		\\
		& \le \frac{1}{|B|^{1+\frac{\b}{Q}}} \int_{B} M_{b}\big( \chi_{_B}\big)(y) dy \notag
		\\
		& \le \frac{2\Psi^{-1}(|B|^{-1})}{|B|^{\frac{\b}{Q}}} \, \|M_{b}\big( \chi_{_B}\big)\|_{L^{\Psi}(B)} \notag
		\\
		& \le \frac{C}{|B|^{\frac{\b}{Q}}} \, \frac{\Psi^{-1}(|B|^{-1})}{\Phi^{-1}(|B|^{-1})} \leq C.
	\end{align*}
	Thus by Lemma \ref{CharLipSp} we get $b\in \dot{\Lambda}_{\beta}(\G)$.
	
	(3) The third statement of the theorem follows from the first and second parts of the theorem.
\end{proof}

If we take $\Phi(t)=t^{p}$ and $\Psi(t)=t^{q}$ with $1\le p<\i$ and $1\le q \le\i$ at Theorem \ref{CommMaxCharOr017}, we have the following result.
\begin{corollary}\label{ZhLpCo}
	Let $0<\beta<1$, $b\in  L^1_{\rm loc}(\G)$, $1< p<q\le \i$ and $\frac{1}{p}-\frac{1}{q}=\frac{\b}{Q}$. Then the condition $b\in \dot{\Lambda}_{\beta}(\G)$ is necessary and sufficient for the boundedness of $M_{b}$ from $L^{p}(\G)$ to $L^{q}(\G)$.
\end{corollary}
\begin{remark} Note that Theorem \ref{CommMaxCharOr017} in the case $\G = \Rn$ were proved in \cite{GulDerHasAMP}.
\end{remark}

The following theorem is valid.
\begin{theorem} \label{CommFrMaxCharOr01W}
	Let $0<\beta<1$, $b\in  L^1_{\rm loc}(\G)$, $\Phi, \Psi$ be Young functions and $\Phi\in\mathcal{Y}$.
	
	$1.~$ If condition \eqref{adFrCharOrl1M} holds,
	then the condition $b\in \dot{\Lambda}_{\beta}(\G)$ is sufficient for the boundedness of $M_{b}$ from $L^{\Phi}(\G)$ to $WL^{\Psi}(\G)$.
	
	$2.~$ If condition \eqref{adFrCharOrl1Mncs} holds and $\frac{t^{1+\varepsilon}}{\Psi(t)}$ is almost decreasing for some $\varepsilon>0$, then the condition $b\in \dot{\Lambda}_{\beta}(\G)$ is necessary for the boundedness of $M_{b}$ from $L^{\Phi}(\G)$ to $WL^{\Psi}(\G)$.
	
	$3.~$ If $\Psi^{-1}(t) \thickapprox \Phi^{-1}(t)t^{-\frac{\b}{Q}}$ and $\frac{t^{1+\varepsilon}}{\Psi(t)}$ is almost decreasing for some $\varepsilon>0$,
	then the condition $b\in \dot{\Lambda}_{\beta}(\G)$ is necessary and sufficient for the boundedness of $M_{b}$ from $L^{\Phi}(\G)$ to $WL^{\Psi}(\G)$.
\end{theorem}
\begin{proof}
	(1) The first statement of the theorem follows from Theorem \ref{AdamsFrMaxCharOrl} and Lemma \ref{pwslip}.
	
	(2) For any fixed ball $B_0$ such that $x\in B_0$ by Lemma \ref{estFrMaxCom} we have $|b(x)-b_{B_0}|\leq M_{b} \chi_{B_0}(x)$.
	This together with the boundedness of $M_{b}$ from $L^{\Phi}(\G)$ to $WL^{\Psi}(\G)$ and Lemma \ref{charorlc}
	\begin{align*}
		|\{x\in B_0: |b(x)-b_{B_0}|> \lambda\}| &\leq |\{x\in B_0: M_{b} \chi_{B_0}(x)> \lambda\}|
		\\
		& \leq \frac{1}{\Psi\left(\frac{ \lambda}{C\|\chi_{B_0}\|_{L^{\Phi}}}\right)} = \frac{1}{\Psi\left(\frac{ \lambda \Phi^{-1}(|B_0|^{-1})}{C}\right)}.
	\end{align*}
	
	Let $t>0$ be a constant to be determined later, then
	\begin{align*}
		\int_{B_0} |b(x)-b_{B_0}|dx & = \int_{0}^{\infty} |\{x\in B_0: |b(x)-b_{B_0}|>\lambda\}| d\lambda
		\\
		& = \int_{0}^{t} \{x\in B_0: |b(x)-b_{B_0}|>\lambda\}| d\lambda
		\\
		& ~~~~ + \int_{t}^{\infty} |\{x\in B_0: |b(x)-b_{B_0}|>\lambda\}| d\lambda
		\\
		& \le t|B_0| + \int_{t}^{\infty} \frac{1}{\Psi\left(\frac{ \lambda \Phi^{-1}(|B_0|^{-1})}{C}\right)} d\lambda
		\\
		& \lesssim  t \, |B_0|+\frac{t}{\Psi\left(\frac{ t \Phi^{-1}(|B_0|^{-1})}{C}\right)},
	\end{align*}
	where we use almost decreasingness of $\frac{t^{1+\varepsilon}}{\Psi(t)}$ in the last step.
	
	Set $t=C|B_0|^{\frac{\b}{Q}}$ in the above estimate, we have
	$$
	\int_{B_0} |b(x)-b_{B_0}|dx\lesssim |B_0|^{1+\frac{\b}{Q}}.
	$$
	Thus by Lemma \ref{CharLipSp} we get $b\in \dot{\Lambda}_{\beta}(\G)$ since $B_0$ is an arbitrary ball in $\G$.
	
	(3) The third statement of the theorem follows from the first and second parts of the theorem.
\end{proof}

If we take $\Phi(t)=t^{p}$ and $\Psi(t)=t^{q}$ with $1\le p<\i$ and $1\le q \le\i$ at Theorem \ref{CommFrMaxCharOr01W}, we have the following result.
\begin{corollary}\label{ZhLpCoW}
	Let $0<\beta<1$, $b\in  L^1_{\rm loc}(\G)$, $1\le p<q\le\i$ and $\frac{1}{p}-\frac{1}{q}=\frac{\b}{Q}$. Then the condition $b\in \dot{\Lambda}_{\beta}(\G)$ is necessary and sufficient for the boundedness of $M_{b}$ from $L^{p}(\G)$ to $WL^{q}(\G)$.
\end{corollary}
\begin{remark}
	
Note that Theorem \ref{CommFrMaxCharOr01W} in the case $\G = \Rn$ were proved in \cite[Corollary 4.6]{GulDerHasAMP}.
\end{remark}

\

\section{Commutators of Fractional Maximal Function in Orlicz Spaces}

For a function $b$ defined on $\G$, we denote
\[
b^{-}(x) :=
\begin{cases}
0 ~, & \text{if } b(x)\geq 0\\
|b(x)|, & \text{if } b(x)<0
\end{cases}
\]
and $b^{+}(x):=|b(x)|-b^{-}(x)$. Obviously, $b^{+}(x)-b^{-}(x)=b(x)$.

The following relations between $[b,M]$ and $M_{b}$ are valid.
Let $b$ be any non-negative locally integrable function. Then for all
$f \in L^1_{\rm loc}(\G)$ and $x\in\G$ the following inequality is valid
\begin{align*}
&\big|[b,M]f(x)\big|  = \big|b(x) Mf(x) - M(bf)(x)\big|
\\
& = \big|M(b(x)f)(x)-M(bf)(x)\big| \le M(|b(x)-b|f)(x) = M_{b}(f)(x).
\end{align*}

If $b$ is any locally integrable function on $\G$, then
\begin{equation}\label{commaxcom}
|[b,M]f(x)|\leq M_{b}(f)(x)+ 2b^{-}(x)Mf(x),\qquad x\in\G
\end{equation}
holds for all $f \in L^1_{\rm loc}(\G)$ (see, for example, \cite{DerGulHasPis2018,ZhWuSun}).

Obviously, the $M_{b}$ and $[b,M]$ operators are essentially different from each other because $M_{b}$ is positive and sublinear and $[b,M]$ is neither positive.

\begin{lemma}\label{LipschCharM02}
	Let $b\in  L^1_{\rm loc}(\G)$ and $\Phi$ be a Young function. Then the following statements are equivalent.
	
	$1.~$ $b\in \dot{\Lambda}_{\beta}(\G)$ and $b \ge 0$.
	
	$2.~$ For all $\Phi\in\Delta_2$ we have
	\begin{align} \label{Ram02D}
		& \sup_{B} |B|^{-\frac{\b}{Q}} \, \Phi^{-1}\big(|B|^{-1}\big) \left\|b(\cdot)-M_{B}(b)(\cdot)\right\|_{L^{\Phi}(B)} \le C.
	\end{align}
	
	$3.~$ There exists $\Phi\in\Delta_2$ such that \eqref{Ram02D} is valid.
\end{lemma}
\begin{proof}
$(1)\Rightarrow (2)$: Let $b$ be any non-negative locally integrable function. Then
\begin{equation}\label{3.1gogmus}
|[b,M] f(x)|\leq M_{b}(f)(x), \qquad x \in \G
\end{equation}
holds for all $f \in L^1_{\rm loc}(\G)$.

By Lemmas \ref{estFrMax} and \ref{pwslip}, for all $x \in B$, we have
\begin{align}\label{Shaf71}
\big| [b,M] \big(\chi_{_{B}}\big)(x) \big|, \big| [|b|,M] \big(\chi_{_{B}}\big)(x) \big| & \le \|b\|_{\dot{\Lambda}_{\beta}} \, M_{\b}\big(b\chi_{_{B}}\big)(x) \notag
\\
& \lesssim  \|b\|_{\dot{\Lambda}_{\beta}} \, |B|^{\frac{\b}{Q}},
\end{align}
and
\begin{align}\label{Shaf72}
\big| [b,M] \big(\chi_{_{B}}\big)(x) \big|,  \big| [|b|,M] \big(\chi_{_{B}}\big)(x) \big| \le \|b\|_{\dot{\Lambda}_{\beta}} \, M_{\b}\big(b\chi_{_{B}}\big)(x) \lesssim  \|b\|_{\dot{\Lambda}_{\beta}} \, |B|^{\frac{\b}{Q}}.
\end{align}

For any fixed $\G$-ball $B$,
\begin{align} \label{Pak01}
I =	& |B|^{-\frac{\b}{Q}} \, \Psi^{-1}\big(|B|^{-1}\big) \left\|b(\cdot)-M_{B}(b)(\cdot)\right\|_{L^{\Psi}(B)}
\end{align}	

By Lemma \ref{estFrMax}, for any $x \in B$,
\begin{align*}
	& b(x) - M_{B}(b)(x) =	\big( b(x) - M_{B}(b)(x) \big)
	\\
	& =  \big( b(x) M_{B} \, \big(\chi_{_{B}}\big)(x) - M\big(b\chi_{_{B}}\big)(x) \big)
	= [b,M] \big(\chi_{_{B}}\big)(x).
\end{align*}
	
Therefore, from \eqref{Shaf71} we obtain
\begin{align}\label{4.2arxiv}
& I = |B|^{-\frac{\b}{Q}} \, \Psi^{-1}\big(|B|^{-1}\big)\|b(\cdot) - M_{B}(b)(\cdot)\|_{L^{\Psi}(B)} \notag
\\
& = |B|^{-\frac{\b}{Q}} \, \Psi^{-1}\big(|B|^{-1}\big) \, \|b(\cdot) M(\chi_{B})(\cdot)-M(b\chi_{B})
(\cdot)\|_{L^{\Psi}(B)} \notag
\\
& = |B|^{-\frac{\b}{Q}} \, \Psi^{-1}\big(|B|^{-1}\big)\|[b,M](\chi_{B})\|_{L^{\Psi}(B)}
\\
& \lesssim |B|^{-\frac{\b}{Q}} \, \Psi^{-1}\big(|B|^{-1}\big) \, \|b\|_{\dot{\Lambda}_{\beta}} \, |B|^{\frac{\b}{Q}} \, \|\chi_{B}\|_{L^{\Psi}} \lesssim \|b\|_{\dot{\Lambda}_{\beta}}. \notag
\end{align}

By \eqref{4.2arxiv}, we get
\begin{align*}
	& |B|^{-\frac{\b}{Q}} \, \Psi^{-1}\big(|B|^{-1}\big) \, \big\|b(\cdot) - M_{B}(b)(\cdot) \big\|_{L^{\Psi}(B)}  \lesssim \|b\|_{\dot{\Lambda}_{\beta}},
\end{align*}
which leads us to \eqref{Ram02D} since $B$ is arbitrary.

$(3)\Rightarrow (1)$: Now, let us prove $b\in \dot{\Lambda}_{\beta}(\G)$ and $b \ge 0$. For any $\G$-ball $B$, let $E = \{y \in B : b(y) \le b_B\}$ and $F = \{y \in B : b(y) > b_B\}$. The following equality is true (see \cite[page 3331]{BasMilRui}):
\begin{align*}
	& \int_{E}|b(y)-b_{B}|dy  = \int_{F}|b(y)-b_{B}|dy.
\end{align*}
Since $b(y) \le b_B \le |b_B| \le M_{B}(b)(y)$ for any $y \in E$, we obtain
\begin{align*}
	& |b(y)-b_{B}| \le \big|b(y) - M_{B}(b)(y)\big|, ~ y \in E .
\end{align*}
Then from Lemma \ref{ShVM01} and \eqref{4.2arxiv} we have
\begin{align*}
	& \frac{1}{|B|^{1+\frac{\b}{Q}}} \int_{B}|b(y)-b_{B}|dy = \frac{2}{|B|^{1+\frac{\b}{Q}}} \int_{E}|b(y)-b_{B}|dy
	\\
	& \le \frac{2}{|B|^{1+\frac{\b}{Q}}} \int_{E} \big|b(y) - M_{B}(b)(y)\big|dy
	\\
	& \le \frac{2}{|B|^{1+\frac{\b}{Q}}} \int_{B} \big|b(y) - M_{B}(b)(y)\big|dy .
\end{align*}
Thus by Lemma \ref{CharLipSp} we get $b\in \dot{\Lambda}_{\beta}(\G)$.

In order to prove $b \ge 0$, it suffices to show $b^{-} = 0$. Observe that $0 \le b^{+}(y) \le |b(y)| \le M_{B}(b)(y)$ for
$y \in B$, therefore, for any $y \in B$, there holds
\begin{align*}
	& 0 \le b^{-}(y) \le M_{B}(b)(y) - b^{+}(y) + b^{-}(y) = M_{B}(b)(y) - b(y).
\end{align*}
Then for any $\G$-ball $B$, we have
\begin{align*}
	\frac{1}{|B|} \int_{B} b^{-}(y) dy & \le \frac{1}{|B|} \int_{B} \big(M_{B}(b)(y) - b(y)\big)dy
	\\
	& = \frac{1}{|B|} \int_{B} \big|b(y) - M_{B}(b)(y)\big|dy
	\\
	& \le \frac{|B|^{\frac{\b}{Q}}}{|B|^{1+\frac{\b}{Q}}} \int_{B} \big|b(y) - M_{B}(b)(y)\big|dy \le C \, |B|^{\frac{\b}{Q}}.
\end{align*}
Let $|B| \to 0$ with $x \in B$. Lebesgue's differentiation theorem assures that
\begin{align*}
	0 \le b^{-}(x) = \lim\limits_{|B| \to 0}\frac{1}{|B|} \int_{B} b^{-}(y) dy &  = 0.
\end{align*}
Thus $b\in \dot{\Lambda}_{\beta}(\G)$ and $b \ge 0$.
The proof of Lemma \ref{LipschCharM02} is completed.
\end{proof}

\begin{theorem}\label{dnndk}
Let $0<\b<1$ and $b$ be a locally integrable function. Suppose that $\Phi, \Psi$ be Young functions, $\Phi\in\mathcal{Y}\cap\nabla_2$ and $\Psi^{-1}(t) \thickapprox \Phi^{-1}(t)t^{-\frac{\b}{Q}}$. Then the following statements are equivalent.
	
$1.~$ $b\in \dot{\Lambda}_{\beta}(\G)$ and $b \ge 0$.
	
$2.~$ $[b,M]$ is bounded from $L^{\Phi}(\G)$ to $L^{\Psi}(\G)$.
	
$3.~$ There exists a constant $C > 0$ such that
\begin{equation}\label{maxbalfx}
\sup_{B} |B|^{-\frac{\b}{Q}} \, \Psi^{-1}\big(|B|^{-1}\big) \, \big\|b(\cdot) - M_{B}(b)(\cdot) \big\|_{L^{\Psi}(B)} \le C.
\end{equation}

$4.~$ There exists a constant $C > 0$ such that
\begin{equation}\label{FrMaxhgFR}
\sup_{B} |B|^{-1-\frac{\b}{Q}} \, \big\|b(\cdot) - M_{B}(b)(\cdot)\big\|_{L^{1}(B)} \le C.
\end{equation}
\end{theorem}

\begin{proof}
Part "$(1) \Leftrightarrow (3)$" and part "$(1) \Leftrightarrow (4)$" follows from Lemma \ref{LipschCharM02}.
	
	$(1)\Rightarrow (2)$: 	
	It follows from \eqref{3.1gogmus} and Theorem \ref{CommMaxCharOr017} that $[b,M_\a]$ is bounded from $L^{\Phi}(\G)$ to $L^{\Psi}(\G)$ since $b\in \dot{\Lambda}_{\beta}(\G)$ and $b \ge 0$.
	
	$(2)\Rightarrow (3)$: For any fixed ball $B \subset \G$ and all $x \in B$, we have (see \cite[pp. 13]{GulResMath2022}).
	$$
	M(\chi_{B})(x) = 1 \qquad \text{and} \qquad M(b\chi_{B})(x)=M_{B}(b)(x).
	$$
	Since $[b,M]$ is bounded from $L^{\Phi}(\G)$ to $L^{\Psi}(\G)$, then
\begin{align}\label{4.2arxivF}
&|B|^{-\frac{\b}{Q}} \, \Psi^{-1}\big(|B|^{-1}\big)\|b(\cdot) - M_{B}(b)(\cdot)\|_{L^{\Psi}(B)} \notag
\\
& = |B|^{-\frac{\b}{Q}}\Psi^{-1}\big(|B|^{-1}\big)\|b(\cdot)M(\chi_{B})(\cdot) - M(b\chi_{B})
(\cdot)\|_{L^{\Psi}(B)} \notag
\\
& = |B|^{-\frac{\b}{Q}}\Psi^{-1}\big(|B|^{-1}\big)\|[b,M](\chi_{B})\|_{L^{\Psi}(B)}
\\
& \leq C |B|^{-\frac{\b}{Q}}\Psi^{-1}\big(|B|^{-1}\big)\|\chi_{B}\|_{L^{\Phi}} \leq C \notag
\end{align}
which implies (3) since the ball $B \subset \G$ is arbitrary.

$(3) \Rightarrow (4)$. We deduce \eqref{FrMaxhgFR} from \eqref{maxbalfx}. Assume \eqref{maxbalfx} holds, then for any fixed $\G$-ball $B$, it follows from Lemma \ref{lemHold} and \eqref{maxbalfx} that
\begin{align*}
	& |B|^{-1-\frac{\b}{Q}} \,  \left\|b(\cdot) - M_{B}(b)(\cdot)\right\|_{L^{1}(B)}
	\\
	& \le 2 \, |B|^{-\frac{\b}{Q}} \, \Psi^{-1}\big(|B|^{-1}\big) \big\|b(\cdot) - M_{B}(b)(\cdot) \big\|_{L^{\Psi}(B)}  \le C \, ,
\end{align*}
where the constant $C$ is independent of $B$. So we obtain \eqref{FrMaxhgFR}.
	
The proof of Theorem \ref{dnndk} is completed.

\end{proof}

If we take $\Phi(t)=t^{p}$ and $\Psi(t)=t^{q}$ with $1\le p<\i$ and $1\le q \le\i$ at Theorem \ref{dnndk}, we have the following result.
\begin{corollary}\label{ZhLpCoxt}
Let $0<\beta<1$, $b\in  L^1_{\rm loc}(\G)$, $b$ be a locally integrable function, $1< p<q\le\i$ and $\frac{1}{p}-\frac{1}{q}=\frac{\b}{Q}$.
Then the following statements are equivalent:
	
$1.~$ $b\in \dot{\Lambda}_{\beta}(\G)$ and $b \ge 0$.
	
$2.~$ $[b,M]$ is bounded from $L^{p}(\G)$ to $L^{q}(\G)$.
	
$3.~$ There exists a constant $C > 0$ such that
\begin{equation*}
\sup_{B} \frac{1}{|B|^{\frac{\b}{Q}}} \, \left(\frac{1}{|B|}\int_{B} \big|b(x) - M_{B}(b)(x)\big|^q dx\right)^{1/q} \le C.
\end{equation*}

$4.~$ There exists a constant $C > 0$ such that
\begin{equation*}
\sup_{B} \frac{1}{|B|^{1+\frac{\b}{Q}}} \int_{B} \big|b(x) - M_{B}(b)(x)\big| dx\le C.
\end{equation*}
\end{corollary}

\begin{remark}
Note that Theorem \ref{dnndk} in the case $\G = \Rn$ and $b \ge 0$ were proved in \cite{GulDerHasAMP} and in the case $\G = \Rn$  in \cite{Zh2017}.
\end{remark}


\begin{theorem}\label{dnndkCommSharpMax}
	Let $0<\b<1$ and $b$ be a locally integrable function. Suppose that $\Phi, \Psi$ be Young functions, $\Phi\in\mathcal{Y}\cap\nabla_2$ and $\Psi^{-1}(t) \thickapprox \Phi^{-1}(t)t^{-\frac{\b}{Q}}$. Then the following statements are equivalent:
	
	$1.~$ $b\in \dot{\Lambda}_{\beta}(\G)$ and $b \ge 0$.
	
	$2.~$ $[b,M^{\sharp}]$ is bounded from $L^{\Phi}(\G)$ to $L^{\Psi}(\G)$.
	
	$3.~$ There exists a constant $C > 0$ such that
	\begin{equation}\label{SharpMaxbalasd}
		\sup_{B} |B|^{-\frac{\b}{Q}} \, \Psi^{-1}\big(|B|^{-1}\big) \, \big\|b(\cdot) - 2M^{\sharp}\big(b \, \chi_{_{B}}\big)(\cdot)\big\|_{L^{\Psi}(B)}\le C.
	\end{equation}
	
	$4.~$ There exists a constant $C > 0$ such that
	\begin{equation}\label{SharpMaxklYh}
		\sup_{B} |B|^{-1-\frac{\b}{Q}} \, \big\|b(\cdot) - 2M^{\sharp}\big(b \, \chi_{_{B}}\big)(\cdot)\big\|_{L^{1}(B)} \le C.
	\end{equation}
\end{theorem}
\begin{proof}
We only need to prove $(1) \Rightarrow (2)$, $(2) \Rightarrow (3)$, $(3) \Rightarrow (4)$ and $(4) \Rightarrow (1)$.

$(1) \Rightarrow (2)$.  Since $b\in \dot{\Lambda}_{\beta}(\G)$ and $b \ge 0$, then for any locally integrable function $f$ and a.e. $x \in \G$
\begin{align*}
& \big|[b,M^{\sharp}]f(x)\big| = \Big| \sup_{B \ni x} \frac{b(x)}{|B|} \, \int_{B} |f(y) - f_{B}|dy
\\
& - \sup_{B \ni x} \frac{1}{|B|} \, \int_{B} |b(y) f(y) - (bf)_{B}|dy \Big|
\\
& \le \sup_{B \ni x} \frac{1}{|B|} \, \int_{B} \big|(b(y)-b(x)) f(y) + b(x) f_{B} - (bf)_{B}\big|dy
\\
& \le \sup_{B \ni x} \Big(\frac{1}{|B|} \, \int_{B} |b(y)-b(x)| \, |f(y)| + \big|b(x) f_{B} - (bf)_{B}\big|  \Big)
\\
& \lesssim \|b\|_{\dot{\Lambda}_{\beta}} \, M_{\b}f(x) + \sup_{B \ni x} \Big| \frac{b(x)}{|B|} \, \int_{B} f(z) dz - \frac{1}{|B|} \, \int_{B} b(z) f(z) dz \Big|
\\
& \lesssim \|b\|_{\dot{\Lambda}_{\beta}} \, M_{\b}f(x) + \sup_{B \ni x} \frac{1}{|B|} \, \int_{B} |b(x)-b(z)| |f(z)| dz
\\
& \lesssim \|b\|_{\dot{\Lambda}_{\beta}} \, M_{\b}f(x).
\end{align*}

Then, it follows from Theorem \ref{AdamsFrMaxCharOrl} that $[b,M^{\sharp}]$ is bounded from $L^{\Phi}(\G)$ to $L^{\Psi}(\G)$.

$(2) \Rightarrow (3)$. Assume $[b,M^{\sharp}]$ is bounded from $L^{\Phi}(\G)$ to $L^{\Psi}(\G)$, we will prove \eqref{SharpMaxbalasd}. For any fixed $\G$-ball $B$, we have (see \cite[page 3333]{BasMilRui} or \cite[page 1383]{ZhWuSun} for details)
\begin{align*}
M^{\sharp} \big(\chi_{_{B}}\big)(x) = \frac{1}{2}  ~~~ \mbox{for all } ~~~ x \in B. 	
\end{align*}

Then, for all $x \in B$,
\begin{align*}
b(x) - 2 M^{\sharp} \big(b \, \chi_{_{B}}\big)(x) & = 2 \Big(\frac{b(x)}{2} - M^{\sharp} \big(b \, \chi_{_{B}}\big)(x)\Big)
\\
& = 2 \Big(b(x) M^{\sharp} \big(\chi_{_{B}}\big)(x)  - M^{\sharp} \big(b \, \chi_{_{B}}\big)(x)\Big)
\\
& = [b,M^{\sharp}] \big(\chi_{_{B}}\big)(x).
\end{align*}

Since $[b,M^{\sharp}]$ is bounded from $L^{\Phi}(\G)$ to $L^{\Psi}(\G)$, then by applying
Lemma \ref{charorlc} and noting that $\Psi^{-1}(t) \thickapprox \Phi^{-1}(t)t^{-\frac{\b}{Q}}$, we have
\begin{align*}
&|B|^{-\frac{\b}{Q}} \, \Psi^{-1}\big(|B|^{-1}\big) \, \big\|b(\cdot) - 2M^{\sharp} \big(b \, \chi_{_{B}}\big)(\cdot)\big\|_{L^{\Psi}(B)} \notag
\\
& = 2 |B|^{-\frac{\b}{Q}}\Psi^{-1}\big(|B|^{-1}\big) \, \big\|[b,M^{\sharp}](\chi_{B})\big\|_{L^{\Psi}(B)}
\\
& \lesssim |B|^{-\frac{\b}{Q}}\Psi^{-1}\big(|B|^{-1}\big) \, \big\|\chi_{B}\big\|_{L^{\Phi}} \lesssim 1. \notag
\end{align*}
which implies \eqref{SharpMaxbalasd}.

$(3) \Rightarrow (4)$: We deduce \eqref{SharpMaxbalasd} from \eqref{SharpMaxklYh}. Assume \eqref{SharpMaxbalasd} holds, then for any fixed $\G$-ball $B$, it follows from Lemma \ref{lemHold} and \eqref{SharpMaxbalasd} that
\begin{align*}
	& |B|^{-1-\frac{\b}{Q}} \,  \left\|b(\cdot) - 2M^{\sharp} \big(b \, \chi_{_{B}}\big)(\cdot)\right\|_{L^{1}(B)}
	\\
	& \le 2 \, |B|^{-\frac{\b}{Q}} \, \Psi^{-1}\big(|B|^{-1}\big) \big\|b(\cdot) - 2M^{\sharp} \big(b \, \chi_{_{B}}\big)(\cdot) \big\|_{L^{\Psi}(B)}  \le C \, ,
\end{align*}
where the constant $C$ is independent of $B$. So we obtain \eqref{SharpMaxklYh}.

$(4) \Rightarrow (1)$. We first prove $b\in \dot{\Lambda}_{\beta}(\G)$. For any fixed $\G$-ball $B$, we have (see (2) in \cite{BasMilRui} for details)
\begin{align}\label{July11}
\big|b_{B}\big| \le 2M^{\sharp} \big(b \, \chi_{_{B}}\big)(x), ~~~ \mbox{for any} ~~~ x \in B.
\end{align}

For any $\G$-ball $B$, let $E = \{y \in B : b(y) \le b_B\}$ and $F = \{y \in B : b(y) > b_B\}$. The following equality is true (see \cite[page 3331]{BasMilRui}):
\begin{align*}
	& \int_{E}|b(y)-b_{B}|dy  = \int_{F}|b(y)-b_{B}|dy.
\end{align*}
Since $b(y) \le b_B \le |b_B| \le 2M^{\sharp} \big(b \, \chi_{_{B}}\big)(y)$ for any $y \in E$, we obtain
\begin{align*}
	& |b(y)-b_{B}| \le \big|b(y) - 2M^{\sharp} \big(b \, \chi_{_{B}}\big)(y)\big|, ~ y \in E .
\end{align*}
Then from Lemma \ref{ShVM01} and \eqref{4.2arxiv} we have
\begin{align*}
	& \frac{1}{|B|^{1+\frac{\b}{Q}}} \int_{B}|b(y)-b_{B}|dy = \frac{2}{|B|^{1+\frac{\b}{Q}}} \int_{E}|b(y)-b_{B}|dy
	\\
	& \le \frac{2}{|B|^{1+\frac{\b}{Q}}} \int_{E} \big|b(y) - 2M^{\sharp} \big(b \, \chi_{_{B}}\big)(y)\big|dy
	\\
	& \le \frac{2}{|B|^{1+\frac{\b}{Q}}} \int_{B} \big|b(y) - 2M^{\sharp} \big(b \, \chi_{_{B}}\big)(y)\big|dy .
\end{align*}

Applying Lemma \ref{CharLipSp} we get $b\in \dot{\Lambda}_{\beta}(\G)$.

In order to prove $b \ge 0$, it suffices to show $b^{-} = 0$.

Then, for all $x \in B$,
\begin{align*}
2M^{\sharp} \big(b \, \chi_{_{B}}\big)(x) - b(x) \ge  \big|b_{B}\big| - b(x) = \big|b_{B}\big| - b^{+}(x) + b^{-}(x).
\end{align*}

By \eqref{SharpMaxklYh}, there exists a constant $C > 0$ such that for any $\G$-ball $B$
\begin{align*}
C & \ge	\frac{1}{|B|^{1+\frac{\b}{Q}}} \int_{B} \big|b(y) - 2M^{\sharp} \big(b \, \chi_{_{B}}\big)(y)\big|dy
\\
& \ge \frac{1}{|B|^{1+\frac{\b}{Q}}} \int_{B} \big(2M^{\sharp} \big(b \, \chi_{_{B}}\big)(y) - b(y)\big)dy
\\
& \ge  \frac{1}{|B|^{1+\frac{\b}{Q}}} \int_{B} \big(\big|b_{B}\big| - b^{+}(y) + b^{-}(y)\big) dy
\\
& = \frac{1}{|B|^{\frac{\b}{Q}}} \Big(\big|b_{B}\big| - \frac{1}{|B|} \int_{B} b^{+}(y) dy + \frac{1}{|B|} \int_{B} b^{-}(y) dy\Big).
\end{align*}
This gives
\begin{align}\label{July78}
\big|b_{B}\big| - \frac{1}{|B|} \int_{B} b^{+}(y) dy + \frac{1}{|B|} \int_{B} b^{-}(y) dy \le C |B|^{\frac{\b}{Q}}
\end{align}
for all balls $B$ and the constant $C$ is independent of $B$.

Let the radius of $\G$-ball $B$ tends to $0$ (then $|B| \rightarrow 0$) with $x \in B$,
Lebesgue differentiation theorem assures that the limit of the left-hand side
of \eqref{July78} equals to
\begin{align*}
|b(x)| - b^{+}(x) + b^{-}(x) = 2b^{-}(x) = 2|b^{-}(x)|.
\end{align*}
Moreover, the right-hand side of \eqref{July78} tends to $0$. Thus, we have $b^{-} = 0$.
Then $b\in \dot{\Lambda}_{\beta}(\G)$ and $b \ge 0$.
The proof of Theorem \ref{dnndkCommSharpMax} is completed.
\end{proof}

\begin{remark}
Theorem \ref{dnndkCommSharpMax} also gives new characterizations of non-negative
Lipschitz functions that they differ from the ones in Theorem \ref{dnndk}.
\end{remark}

If we take $\Phi(t)=t^{p}$ and $\Psi(t)=t^{q}$ with $1\le p<\i$ and $1\le q \le\i$ at Theorem \ref{dnndkCommSharpMax}, we have the following result.
\begin{corollary}\label{ZhSharpCoxtFGH}
Let $0<\beta<1$, $b\in  L^1_{\rm loc}(\G)$, $b$ be a locally integrable function, $1< p<q\le\i$ and $\frac{1}{p}-\frac{1}{q}=\frac{\b}{Q}$.
Then the following statements are equivalent:
	
$1.~$ $b\in \dot{\Lambda}_{\beta}(\G)$ and $b \ge 0$.
	
$2.~$ $[b,M_\alpha]$ is bounded from $L^{p}(\G)$ to $L^{q}(\G)$.
	
$3.~$ There exists a constant $C > 0$ such that
\begin{equation*}
\sup_{B} \frac{1}{|B|^{\frac{\b}{Q}}} \, \left(\frac{1}{|B|}\int_{B} \big|b(x) - 2M^{\sharp}\big(b \, \chi_{_{B}}\big)(x)\big|^q dx\right)^{1/q} \le C.
\end{equation*}
	
$4.~$ There exists a constant $C > 0$ such that
\begin{equation*}
\sup_{B} \frac{1}{|B|^{1+\frac{\b}{Q}}} \int_{B} \big|b(x)-2M^{\sharp}\big(b \, \chi_{_{B}}\big)(x)\big| dx\le C.
\end{equation*}
\end{corollary}

\


\

{\bf \Large Acknowledgements} ~

\

The research of author was partially supported by grant of Cooperation Program 2532 TUBITAK - RFBR (RUSSIAN foundation for basic research) (Agreement number no. 119N455). 

\

{\bf \large Data Availibility}   Data is contained within the article.

\

{\bf \large Conflicts of interest} The author states that there is no conflict of interest.

\

\

\end{document}